\documentclass[submission]{our_dmtcs}
\usepackage{geometry}                
\usepackage[latin1]{inputenc}
\usepackage{subfigure}

\usepackage{amssymb}
\usepackage{epstopdf}
\usepackage[square,sort,comma,numbers]{natbib}

\usepackage{amsmath}
\usepackage{amscd}
\usepackage{mathrsfs}
\usepackage{amscd}

\newtheorem{thm}{Theorem}
\newtheorem{lem}{Lemma}

\newtheorem{prop}[thm]{Proposition}

\renewcommand{\Re}{\mathbb{R}}

\newcommand{\dd}{\delta_{\mathbf{diam}}}

\renewcommand{\Re}{{\mathbb{R}}}

\author{David Bryant\addressmark{1}\thanks{Email: \email{david.bryant@otago.ac.nz}}
\and Paul F. Tupper\addressmark{2}\thanks{Email: \email{pft3@math.sfu.ca}}}
\title{Diversities and the Geometry of Hypergraphs}
\address{\addressmark{1}Allan Wilson Centre, Dept.\ of Mathematics and Statistics, University of Otago, Dunedin 9054, New Zealand.\\
\addressmark{2}Dept.\ of Mathematics, Simon Fraser University. 8888 University Drive, Burnaby, British Columbia V5A 1S6, Canada.}

\keywords{Diversities, Metric embedding, Fractional Steiner Tree Packing, Hypergraphs}

\received{2013-12-4}
\revised{\today}

\newcommand{\sT}{\mathscr{T}}

\newcommand{\vv}{\mathbf{v}}

\newcommand{\diam}{\mathrm{diam}}

\newcommand{\conv}[1] {\mathrm{conv}(#1)}

\renewcommand{\Cap}{\mathrm{Cap}}
\newcommand{\Dem}{\mathrm{Dem}}

\begin{document}
\maketitle

\begin{abstract}


The embedding of finite metrics in $\ell_1$ has become a fundamental tool for both combinatorial optimization and large-scale data analysis. One important  application is to network flow problems as there is close relation between max-flow min-cut theorems and the minimal distortion embeddings of metrics into $\ell_1$.  Here we show that this theory can be generalized to a larger set of combinatorial optimization problems on both graphs and hypergraphs. This theory is not built on metrics and metric embeddings, but on {\em diversities}, a type of multi-way metric introduced recently by the authors. We explore diversity embeddings, $\ell_1$ diversities, and their application to  {\em Steiner Tree Packing} and {\em Hypergraph Cut} problems.

\end{abstract}

\section{Introduction}

In their influential paper ``The Geometry of Graphs and its Algorithmic Applications",  Linial et al.~\cite{Linial95} introduce
 a novel and powerful set of techniques to the algorithm designer's toolkit.  They show how to use the mathematics of metric embeddings to help solve difficult problems in combinatorial optimization. The approach  inspired a large body of further work on metric 
embeddings and  their applications.
%
%

Our objective here is to show how this extensive body of work might be generalized to  the {\em geometry of hypergraphs}.
Recall that a {\em hypergraph} $H=(V,E)$ consists of a set of vertices $V$ and a set of {\em hyperedges} $E$, where each $A \in E$ is a subset of $V$.
 The underlying geometric objects in this new context will not be metric spaces, but {\em diversities}, a generalization of metrics recently introduced by Bryant and Tupper \cite{Bryant12}. Diversities are a form of multi-way metric which have already given rise to a substantial, and novel, body of theory \cite{Bryant12,Herrmann12,Poelstra13,Espinola14}. We hope to demonstrate that a switch to diversities opens up a whole new array of problems and potential applications, potentially richer than that for metrics. 

The result of \cite{Linial95} which is of particular significance to us is  the use of metric embeddings to bound the difference between cuts and flows in a {\em multi-commodity flow} problem. Let $G=(V,E)$ be a graph with a non-negative edge capacity $C_{uv} \geq 0$ for every edge $\{u,v\} \in E$. We are given a set of {\em demands} $D_{uv} \geq 0$ for $u,v \in V$. The objective of the { multi-commodity flow} problem is to find the largest value of $f$ such that we can simultaneously flow at least $f \cdot D_{uv}$ units between $u$ and $v$ for all $u$ and $v$. As usual,
the total amount of flow along an edge cannot exceed its capacity.

Multi-commodity flow is a linear programming problem (LP) and  can be solved in polynomial time. The {\em dual} of the LP is a relaxation of a min-cut problem which generalizes several NP-hard graph partition problems. Given $S \subseteq V$ let $\Cap(S)$ be the sum of edge capacities of edges joining $S$ and $V \setminus S$ and let $\Dem(S)$ denote the sum of the demands for pairs $u,v$ with $u \in S$ and $v \in V \setminus S$. We then have $f \leq \frac{\Cap(S)}{\Dem(S)}$ for every  $S \subseteq V$. When there is a single demand, the minimum of $\frac{\Cap(S)}{\Dem(S)}$ equals the maximum value of $f$,  a consequence of the max-flow min-cut theorem.  In general, for more than one demand there will be a gap between the values of the minimum cut and the maximum flow. Linial et al \cite{Linial95}, building on the work of \cite{Leighton88}, show that this gap can be bounded by the {\em distortion} required to embed a particular metric $d$ (arising from the LP dual) into $\ell_1$ space. The metric $d$ is \emph{supported} on the graph $G(V,E)$, meaning that it is the shortest path metric for some weighting of the edges $E$.  By applying the extensive literature on distortion bounds for metric embeddings they obtain new approximation bounds for the min-cut problem.

In this paper we consider generalizations of  the multi-commodity flow and corresponding minimum cut problems. A natural generalization of the single-commodity maximum flow problem in a graph is \emph{fractional Steiner tree packing} \cite{Jain03}. Given a graph $G=(V,E)$ with weighted edges, and a subset $S \subseteq V$, find the maximum total weight of trees in $G$ spanning $S$ such 
that the sum of the weights of trees containing an edge does not exceed the capacity of that edge. Whereas multi-commodity flows are typically used to model transport of physical substances (or vehicles), the Steiner tree packing problem arises from models of information, particularly the broadcasting of information (see \cite{Lau07} for references).

The fractional Steiner tree packing problem generalizes further to incorporate multiple commodities, a formulation which occurs naturally in  multicast and VLSI design applications (see \cite{Saad08}). For each $S \subseteq V$ we have a demand $D_S$ (possibly zero) and the set $\sT_S$ of trees in the graph spanning $S$. A {\em generalized flow} in this context is an assignment of non-negative weights $z_{t,S}$ to the trees in $\sT_S$ for all $S$, with the constraint that for each edge, the total weight of trees including that edge does not exceed the edge's capacity. The objective is to find the largest value of $f$ for which there is a flow with weights $z_{t,S}$ satisfying 
\begin{equation} \label{HSTconstraint}
\sum_{t \in \sT_S} z_{t,S} \geq f \cdot D_S
\end{equation}
for all demand sets $S$. 

These problems translate directly to hypergraphs, permitting far more complex relationships between the different capacity constraints. As  for graphs, we have demands $D_S$ defined for all $S \subseteq V$. Each hyperedge $A \in E$ has a non-negative capacity. We let $\sT_S$ denote the set of all {\em minimal connected sub-hypergraphs} which include $S$ (not necessarily trees). A flow in this context is an assignment of non-negative weights $z_{t,S}$ to the trees in $\sT_S$ for all $S$, with the constraint that for each hyperedge, the total weight of trees including that hyperedge does not exceed the hyperedge's capacity. As in the graph case, the aim is determine the largest value of $f$ for which there is a flow with weights $z_{t,S}$ satisfying the constraint \eqref{HSTconstraint} for all demand sets $S$.

All of these generalizations of the  multi-commodity flow problem have a dual problem that is a relaxation of a  corresponding min-cut problem. For convenience, we assume any missing edges or hyperedges are included with capacity zero. 
For a subset $U \subseteq V$ let $\partial U$  be the set of edges or hyperedges which have endpoints in both $U$ and $V \setminus U$. The min-cut problem in the case of graphs is to find the cut $U$ minimizing
\[\frac{\sum\limits_{e \in \partial U} C_e}{\sum\limits_{ S \in \partial U } D_S}\]
where $e$ runs over all pairs of distinct vertices in $V$,
while in hypergraphs we find $U$ which minimizes
\[\frac{\sum\limits_{A \in \partial U} C_A}{\sum\limits_{B \in \partial U} D_B},\]
where $A, B$ run over all subsets of $V$.

In both problems the value of a min-cut is an upper bound for corresponding value of the maximum flow. Linial et al.\ \cite{Linial95} showed that the ratio between the min-cut and the max flow can be bounded using metric embeddings. Our main result is that this relationship generalizes to the fractional Steiner problem  with multiple demand sets, on both graphs and hypergraphs, once we consider diversities instead of metrics. The following theorems depend on the notions of diversities being supported on hypergraphs and $\ell_1$-embeddings of diversities, which we will define in subsequent sections.

\begin{thm} \label{major1}
Let $H = (V,E)$ be a hypergraph. Let $\{C_A\}_{A \in E}$ be a set of edge capacities and $\{D_S\}_{S \subseteq V}$ a set of demands. There is a diversity $(V,\delta)$ supported on $H$, such that the ratio of the min-cut to the maximum (generalized) flow for the hypergraph is bounded by the minimum distortion embedding of $\delta$ into $\ell_1$.
\end{thm}

Gupta et al.\ \cite{Gupta04} proved a converse of the result of Linial et al.\ by showing that, given any graph $G$ and metric $d$ supported on it, we could determine capacities and demands so that the bound given by the minimal distortion embedding of  $d$ into $\ell_1$ was tight. We establish the analogous result for the generalized flow problem in hypergraphs.

\begin{thm} \label{major2}
Let $H = (V,E)$ be a hypergraph, and let $\delta$ be a diversity supported on it. There is a set $\{C_A\}_{A \in E}$ of edge capacities 
and a set $\{D_S\}_{S \subseteq V}$ of demands so that the ratio of the min-cut to the maximum (generalized) flow equals the distortion of the minimum distortion embedding of $\delta$ into $\ell_1$. \end{thm}

A major benefit of the link between min-cut and metric embeddings was that Linial et al.\ and others could make use of an extensive body of work on metric geometry to establish improved approximation bounds. In our context, the embeddings of diversities is an area which is almost completely unexplored. We prove a few preliminary bounds here, though much work remains. \\

The structure of this paper is as follows.  We begin in Section 2 with a brief review of diversity theory, including a list of examples of diversities. In Section 3 we focus on $L_1$ and $\ell_1$ diversities, which are the generalizations of $L_1$ and $\ell_1$ metrics. These diversities arise in a variety of different contexts. Fundamental properties of $L_1$ diversities are established, many of which closely parallel results on metrics. 

In Section 4 we show how the concepts of metric embedding and distortion are defined for diversities, and establish a range of preliminary bounds for distortion and dimension. Finally, in Section 5, we 
prove the analogues of Linial et al's \cite{Linial95} and Gupta et al's \cite{Gupta04} results on multi-commodity flows, as stated in Theorems~\ref{major1} and \ref{major2} above.

\section{Diversities}


A {\em diversity} is a pair $(X,\delta)$ where $X$ is a set and $\delta$ is a function from the finite subsets of $X$ to $\Re$ satisfying  
\begin{quotation}
\noindent (D1) $\delta(A) \geq 0$, and $\delta(A) = 0$ if 
and only if 
$|A|\leq 1$. \\
(D2) If $B \neq \emptyset$ then $\delta(A\cup B) + \delta(B \cup C) \geq \delta(A \cup C)$
\end{quotation}
for all finite $A, B, C \subseteq X$.  Diversities are, in a sense, an extension of the metric concept. Indeed, every diversity has an {\em induced metric}, given by $d(a,b) = \delta(\{a,b\})$ for all $a,b \in X$. Note also that  $\delta$ is {\em monotonic}: $A \subseteq B$ implies $\delta(A) \leq \delta(B)$. For convenience, in the remainder of the paper we will relax condition (D1) and allow $\delta(A) = 0$ even when $|A|>1$.  Likewise, for metrics we allow $d(x,y)=0$ even if $x \neq y$.

%

We  define embeddings and distortion for diversities in the same way as for metric spaces. Let $(X_1,\delta_1)$ and $(X_2,\delta_2)$ be two diversities and suppose $c \geq 1$. A map $\phi:X_1 \rightarrow X_2$ has {\em distortion} $c$ if there is $c_1,c_2 > 0$ such that $c = c_1 c_2$ and
\[\frac{1}{c_1} \delta_1(A) \leq \delta_2(\phi(A)) \leq c_2 \delta_1(A)\]
for all finite $A \subseteq X_1$. We say that $\phi$ is an {\em isometric embedding} if it has distortion $1$ and an {\em approximate embedding} otherwise. 

\subsection{Examples of diversities} \label{sec:DiversityList}

Bryant and Tupper \cite{Bryant12} provide several examples of diversities. We expand that list here. 

\begin{enumerate}
\item 
\noindent {\em Diameter diversity.}  Let $(X,d)$ be a metric space.  For all finite $A \subseteq X$ let 
\[
\dd (A) = \max_{a,b \in A} d(a,b) = \diam(A).
\]
\item {\em $\ell_1$ diversity.} Let $\ell^m_1$ denote the diversity $(\Re^m, \delta_1)$, where \[\delta_1(A) = \sum_i \max_{a,b} \{|a_i - b_i |:a,b \in A\}\]
for all finite $A \subset \Re^m$.
\item {\em $L_1$ diversity.} Let $(\Omega,\mathcal{A},\mu)$ be a measure space and let $L_1$ denote the set of all all measurable functions $f : \Omega \rightarrow \Re$ with $\int_\Omega |f(\omega)| d\mu(\omega) < \infty$. An $L_1$ diversity is a pair $(L_1,\delta_1)$ where $\delta_1(F)$ is given by
\[\delta_1(F) = \int_\Omega \max \{|f(\omega) - g(\omega)|:f,g \in F\}\,\, d \mu(\omega)\]
for all finite $F \subseteq L_1$. To see that $(L_1,\delta_1)$  satisfies (D2), consider the triangle inequality for the diameter diversity on a real line and integrate over $\omega$. 
\item {\em Phylogenetic diversity.}
Let $T$ be a phylogenetic tree with taxon set $X$. For each finite $A \subseteq X$, the {\em phylogenetic diversity} of $A$ is the length $\delta(A)$ of the smallest subtree of $T$ connecting taxa in $A$ \cite{Faith92,Steel05,Minh09,BryantKlaere}.
\item {\em Steiner diversity.} 
Let $(X,d)$ be a metric space. For each finite $A \subseteq X$ let $\delta(A)$ denote the minimum length of a Steiner tree connecting elements in $A$. 
\item {\em Hypergraph Steiner diversity.}
Let $H = (X,E)$ be a hypergraph and let $w:E \rightarrow \Re_{ \geq 0}$ be a non-negative weight function. Given  $A \subseteq X$ let $\delta(A)$ denote the minimum of $w(E'):=\sum_{e \in E'} w(e)$ over all subsets $E' \subseteq E$ such that the sub-hypergraph induced by $E'$ is connected and includes $A$. Then $(X,\delta)$ is a diversity.
\item {\em Measure diversity.} 
Let $(M,\Sigma,\mu)$ be a measure space, where $\Sigma$ is a $\sigma$-algebra of subsets of $M$ and $\mu \colon \Sigma \rightarrow [0,\infty]$.  Let $X$ be the collection of all sets in $\Sigma$ with finite measure.  
For sets $E_1, E_2, \ldots, E_k \in X$ we let
\begin{equation}
\delta(\{E_1, E_2, \ldots, E_k\}) = \mu\left( \bigcup_{i=1}^k E_i \right) - \mu \left(\bigcap_{i=1}^k E_i \right) = \mu\left(\bigcup_{i=1}^k E_i  \setminus \bigcap_{i=1}^k E_i \right). \label{MeasureDiv} \end{equation}
\item {\em Smallest Enclosing Ball diversity}. Let $(X,d)$ be a metric space. For each finite $A \subseteq X$ let $\delta(A)$ be the diameter of the smallest closed ball containing $A$. Note that if that every pair of points in $(X,d)$ are connected by a  geodesic then $(X,d)$ will be the induced metric of $(X,\delta)$, though this does not hold in general. 
\item {\em Travelling Salesman diversity.} Let $(X,d)$ be a metric space. For every finite $A \subseteq X$, let $\delta(A)$ be the minimum of 
\[\frac{1}{2} \left( d(a_1,a_2) + d(a_2,a_3) + \cdots +  d(a_{|A|},a_1) \right)\]
over all orderings $a_1,a_2,a_3,\ldots,a_{|A|}$ of $A$.
\item {\em Mean-width diversity}.  We define the \emph{mean-width diversity} for  finite $A \subset \Re^n$
as the mean width of $\conv{A}$, the convex hull of $A$, suitably scaled. Specifically,  given a compact convex set $K \subset \Re^n$
 and unit vector $u \in \Re^n$,
the width of $K$  in direction $u$ is given by
\[
w(K,u) = \max_{a \in K} a \cdot u  - \min_{a \in K} a \cdot u.
\]
That is, $w(K,u)$ is the minimum distance between two hyperplanes with normal $u$ which enclose $K$. The {\em mean width} of $K$ is given by
\[m_n(K) = \frac{1}{\mu_{n-1}(S^{n-1})} 
\int_{S^{n-1}} w(K,u)\, d\mu_{n-1}(u),\]
where $\mu_{n-1}$ denotes the surface measure on the unit sphere $S^{n-1}$ \cite{Taylor06}. Shephard \cite{Shephard68} observed that the mean width varies according to the space that $K$ is sitting in, whereas a scaled {\em absolute mean width}
\[M(K) = \frac{1}{\mu_n(S^{n})} \int_{S^{n-1}} w(K,u) \, d\mu_{n-1}(u)\]
does not. A simple calculation gives that for points $a,b$ in $\Re$ we have $M([a,b]) = \frac{1}{\pi} |a-b|$. Hence we define the {\em mean-width diversity} 
\begin{equation}
\label{eq:widthConstant}
\delta_w(A) = \pi M(\conv{A}) = \frac{\pi \mu_{n-1}(S^{n-1})}{\mu_n(S^n)} m_n(\conv{A}) = \frac{\pi}{B(n/2,1/2)} m_n(\conv{A}),
\end{equation}
so that the induced metric of $\delta_w$ is the Euclidean metric.  Here  $B(\cdot,\cdot)$ is the beta function.
Note that $\frac{\pi}{B(n/2,1/2)} = \sqrt{\frac{\pi}{2}} n^{1/2} + o(\frac{1}{\sqrt{n}})$, see \cite{Abramowitz64}.
\item {\em $S$-diversity}. Let $X$ be a collection of random variables taking values in the same state space. For every finite $A = \{A_1,\ldots,A_k\} \subseteq X$ let $\delta(A)$ be the probability that $A_1,A_2,\ldots,A_k$ do not all have the same state. Then $(X,\delta)$ is a diversity, termed the {\em $S$-diversity} since $S$, the proportion of segregating (non-constant) sites, is a standard measure of genetic diversity in an alignment of genetic sequences (see, e.g. \cite{BryantKlaere}). 
\end{enumerate}

Below, we will show that $\ell_1$ diversities, phylogenetic diversities, measure diversities, mean-width diversities and $S$-diversities are all examples of $L_1$-embeddable diversities.

\subsection{Extremal diversities}

In metric geometry we say that one metric {\em dominates} another on the same set if distances under the first metric are all greater than, or equal to, distances under the second. The relation forms a partial order on the cone of metrics for a set:
given any two metric spaces $(X,d_1)$ and $(X,d_2)$ we write $d_1 \preceq d_2$ if  $d_1(x,y) \leq d_2(x,y)$ for all $x,y \in X$. The partial order $\preceq$ provides a particularly useful characterization of the standard shortest-path graph metric $d_G$. Let $G = (V,E)$ be a graph with edge weights $w: E \rightarrow \Re_{\geq 0}$. The shortest path metric $d_G$ is then the unique, maximal metric (under $\preceq$) which satisfies 
$d(u,v) \leq w(\{u,v\})$ for all $\{u,v\} \in E$. Given that the {\em geometry of graphs} of  \cite{Linial95} is based on the shortest path metric, it is natural to explore what arises when we apply the same approach to diversities.

We say that a diversity $(X,\delta_2)$ dominates another diversity $(X,\delta_1)$   if $\delta_1(A) \leq \delta_2(A)$ for all finite $A \subseteq X$, in which case we write $\delta_1 \preceq \delta_2$. Applying these to graphs, and hypergraphs, we obtain the diversity analogue to the shortest-path metric.

\begin{thm} \label{thm:SteinMax}
\begin{enumerate}
\item
Let $G = (V,E)$ be a graph with non-negative weight function $w:E \rightarrow \Re_{\geq 0}$. The Steiner tree diversity is the unique maximal diversity $\delta$ such that $\delta(\{u,v\}) \leq w(\{u,v\})$ for all $\{u,v\} \in E$.
\item
Let $H = (V,E)$ be a hypergraph with non-negative weight function $w:E \rightarrow \Re_{ \geq 0}$. The hypergraph Steiner diversity is the unique maximal diversity $\delta$ such that $\delta(A) \leq w(A)$ for all $A \in E$.
\end{enumerate}
\end{thm}
\begin{proof}
Note that 1.\ is a special case of 2. We prove 2.

Let $\delta_H$ denote the hypergraph Steiner diversity for $H$.  For any edge $A$, the edge itself forms a connected sub-hypergraph, so $\delta_H(A) \leq w(A)$. Let $\delta$ be any other diversity which also satisfies  $\delta(A) \leq w(A)$ for all $A \in E$. For all $B \subseteq V$ there is $E' \subseteq E$ such that the sub-hypergraph induced by $E'$ is connected, contains $B$, and has summed weight $\delta_H(B)$. Multiple applications of the triangle inequality (D2) gives
\[\delta(B) \leq \sum_{A \in E'} \delta(A) \leq \sum_{A \in E'} w(A) = \delta_H(B).\]
\end{proof}

As a further consequence, we can show that the hypergraph Steiner diversity dominates all diversities with a given induced metric.

\begin{thm} \label{DivBounds}
Let $(X,\delta)$ be a diversity with induced metric space $(X,d)$. Let $\dd$ denote the diameter diversity on $X$ and let $\delta_S$ denote the Steiner diversity on $X$. Then for all finite $A \subseteq X$,
\[\dd(A) \leq \delta(A) \leq \delta_S(A) \leq (|A|-1) \dd(A).\]
\end{thm}
\begin{proof}
If $|A| \leq 1$ then $\dd(A) = \delta(A) = \delta_S(A) = 0$. Suppose $2\leq |A| <\infty$. 
There is $a,a'$ such that 
\[\dd(A) = d(a,a') = \delta(\{a,a'\}) \leq \delta(A),\]
the last inequality following from the monotonicity of $\delta$. Let $G$ be the complete graph with vertex set $A$ and edge weights $w(\{a,a'\}) = d(a,a')$. Then $\delta(A) \leq \delta_S(A)$ by Theorem~\ref{thm:SteinMax}. To obtain the final inequality, consider any ordering of the elements of $A$: $a_1, a_2, \ldots, a_{|A|}$. Then, using the triangle inequality repeatedly gives
\[
\delta(A) \leq \delta(\{a_1,a_2\}) + \delta(\{a_2,a_3\}) + \cdots +\delta(\{a_{|A|-1},a_{|A|}\}) \leq (|A|-1) \dd(A).
\]
\end{proof}

\section{$L_1$-embeddable diversities}

\subsection{General Properties}

$L_1$ diversities were defined in Section~\ref{sec:DiversityList}. 
We say that a diversity $(X,\delta)$ is $L_1$-embeddable if there exists an isometric embedding of $(X,\delta)$ into an $L_1$ diversity. A direct consequence of the definition of $L_1$ diversities (and the direct sum of measure spaces) is that if $(X,\delta_1)$ and $(X,\delta_2)$ are both $L_1$ diversities then so are $(X,\delta_1 + \delta_2)$ and $\lambda \delta_1$ for $\lambda > 0$.  Hence the $L_1$-embeddable diversities on a given set form a cone.

Deza and Laurent \cite{Deza97} make a systematic study of the identities and inequalities satisfied by the cone of $L_1$ {\em metrics}. Much of this work will no doubt have analogues in diversity theory. For one thing, every identity for $L_1$ metrics is also an identity for the induced metrics of $L_1$ diversities. However $L_1$ diversities will satisfy a far richer collection of identities. One example is the following.

\begin{prop} \label{prop:circleL1}
Let $(X,\delta)$ be $L_1$-embeddable and let $A_1, \ldots, A_n$ be finite subsets of $X$ with union $A$. Then
\begin{equation}
\delta(A) \leq \frac{1}{2} \left( \delta(A_1 \cup A_2) + \delta(A_2 \cup A_3) + \cdots + \delta(A_n \cup A_1)\right). \label{eqn:circleL1}
\end{equation}
\end{prop}
\begin{proof}
First suppose $(X,\delta)$ embeds isometrically in $\ell_1^1$,  the diameter diversity on $\Re$. Let $x_m$ and $x_M$ be the minimum and maximum elements in $A$. Identify $A_{n+1}$ with $A_1$ and $A_0$ with $A_n$. There is $i,j$ such that $x _m \in A_i$, $x_M \in A_j$ and, without loss of generality, $i \leq j$. 
If $i=j$ then 
\[\delta(A) \leq \delta(A_i) \leq \frac{1}{2} (\delta(A_i,A_{i-1}) + \delta(A_i,A_{i+1})).\]
If $i \neq j$ then, without loss of generality, $i<j$. Select $y_1,\ldots,y_n$ such that $y_i = x_m$, $y_j = x_M$ and $y_k \in A_k$ for all $k$. Then, considering two different paths from $y_i$ to $y_j$ we obtain
\begin{align*}
|y_i-y_j| & \leq |y_{i+1} - y_i| + |y_{i+2} - y_{i+1} | + \cdots + |y_{j} - y_{j-1}| \\
\intertext{and}
|y_i-y_j| & \leq |y_{i} - y_{i-1}| + \cdots + |y_2 - y_1| + |y_1 - y_n|+ |y_n - y_{n-1}| + \cdots + |y_{j+1} - y_{j} | \\
\end{align*}
so 
\begin{align*}
\delta(A) & = |x_M - x_m|   = |y_i-y_j|  \leq \frac{1}{2} \sum_{i=1}^{n-1} |y_{i+1} - y_1| + |y_n-y_1| 
 \leq \frac{1}{2} (\delta(A_i,A_{i-1}) + \delta(A_i,A_{i+1})).
\end{align*}
The case for general $L_1$-embeddable  diversities can be obtained by integrating this inequality over the measure space.
\end{proof}

Esp\'{i}nola and Pi{\c{a}}tek \cite{Espinola14} investigated when hyperconvexity for diversities implied hyperconvexity for their induced metrics, proving that this held whenever the induced metric $(X,d)$ of a diversity satisfies
\begin{equation} \label{eq:EPcondition}
(|A| - 1) \cdot \delta(A) \leq \sum_{1 \leq i<j\leq k} d(a_i,a_j)
\end{equation}
for all $A = \{a_1,\ldots,a_k\} \subseteq X$. (See \cite{Espinola14} for definitions and results). A consequence of Proposition~\ref{prop:circleL1} is that this property holds for all $L_1$-embeddable diversities.

\begin{prop}
If $(X,\delta)$ is $L_1$ embeddable then its induced metric $(X,d)$ satisfies \eqref{eq:EPcondition} for all finite $A \subseteq X$.
\end{prop}
\begin{proof}
Suppose that $|A| = k$. There are $(k-1)!/2$ cycles of length $k$ through $A$, and each edge is contained in exactly $ (k-2)!$ such cycles. For each cycle $a_{\sigma(1)},a_{\sigma(2)},\ldots,a_{\sigma(k)}$ we have from Proposition~\ref{prop:circleL1} that
\[ d(a_{\sigma(k)},a_{\sigma(1)})  + \sum_{i=1}^{k-1} d(a_{\sigma(i)},a_{\sigma(i+1)})   \geq 2 \delta(A).\]
Hence
\begin{align*}
 \sum_{1 \leq i<j\leq k} d(a_i,a_j) & = \frac{1}{(k-2)!} \sum_\sigma \left(d(a_{\sigma(k)},a_{\sigma(1)})  + \sum_{i=1}^{k-1} d(a_{\sigma(i)},a_{\sigma(i+1)})   \right) \\
& \geq \frac{(k-1)!}{2(k-2)!} 2 \delta(A) \\
& = (k-1) \delta(A).
\end{align*}
\end{proof}

\subsection{Examples of $L_1$-embeddable diversities}

We now examine three examples of diversities $(X,\delta)$ which  are $L_1$-embeddable. In all three cases, the diversity need not be finite, nor even finite dimensional. Later, we examine $L_1$-embeddable diversities for finite sets.

\begin{prop}
Measure diversities, $S$-diversities and mean-width diversities are all $L_1$-embeddable.
\end{prop}
\begin{proof}
We treat each kind of diversity in turn. \\
{\em Measure diversities}. \\ In a measure diversity any element  $A \in \Sigma$ can be naturally identified with the function $\mathbf{1}_A$ in $L_1(\Omega,\Sigma,\mu)$. Observe now that 
\begin{eqnarray*}
\delta(\{A_1, \ldots, A_n\}) & = & \mu( \cup_i A_i \setminus \cap_i A_i) \\
& = & \int_\Omega ( \max_i \mathbf{1}_{A_i}(\omega) - \min_i \mathbf{1}_{A_i}(\omega) ) \,d\mu(\omega)\\
& = & \int_\Omega \diam_i \{ \mathbf{1}_{A_i}(\omega)\} \, d\mu(\omega) \\
&= & \delta_1( \mathbf{1}_{A_1}, \ldots, \mathbf{1}_{A_n}).
\end{eqnarray*}
{\em Mean-width diversities. } \\ Let $(\Re^n,\delta_w)$ be the $n$-dimensional mean-width diversity. Consider $L_1(S^{n-1},\mathcal{B},\nu)$ where $S^{n-1}$ is the unit sphere in $\Re^n$, $\mathcal{B}$ is the Borel subsets of $S^{n-1}$ and $\nu$ is the measure given by $\nu(B) = \pi \mu_{n-1}(B)/ \mu_n(S^n)$ for all $B \in \mathcal{B}$ where $\mu_{n-1}$ is the surface measure on $S^{n-1}$. Let $\Phi(a)$ for $a \in \Re^k$ be the function $f_a (\vv) = a \cdot \vv$ for $\vv \in S^{k-1}$. Then
\[
\delta_w( a_1, \ldots, a_n ) = \int_{S^{k-1}} \diam( f_{a_1}(v) ,\ldots, f_{a_n}(v) ) \, d\nu(v).
\]
 Thus $(\Re^k, \delta_w)$ is embedded in $L_1$.
\\
{\em $S$-diversities}.\\ Let $(X,\delta)$ be an $S$-diversity. Suppose that the random variables in $X$ have state space $\mathcal{S}$ and that they are defined on the same probability space $(\Omega,\Sigma,\mu)$.  For each $X_\gamma \in X$ let $f_\gamma \colon \mathcal{S} \times \Omega \rightarrow \Re$ be given by $f_\gamma(s,\omega)=1$ if $X_\gamma(\omega)=s$ and $0$ otherwise. Then 
\begin{eqnarray*}
\delta_1( f_{\gamma_1},\ldots, f_{\gamma_k}) & = & \int_{S \times \Omega} \diam ( f_{\gamma_1}(s,\omega),\ldots, f_{\gamma_k}(s,\omega) )\, \, d\nu(s) \times d\mu(\omega) \\
& = & \mathbb{P} \{ X_{\gamma_i} \neq X_{\gamma_j} \mbox{ for some } i, j \}
\end{eqnarray*}
\end{proof}

In the case of measure diversities, we can also prove a converse result, in the sense that every $L_1$ diversity can be embedded in a measure diversity. We first make some observations about $\Re$. Consider the map $\phi \colon \Re \rightarrow \mathcal{P}(\Re)$ given by 
\[
\phi(x) = \left\{ \begin{array}{ll}
 \left[ 0, x \right] &  \mbox{if $x \geq 0$}; \\
 \left[ x, 0 \right)  &  \mbox{if $x < 0$}.
\end{array}
\right.
\]
Note that $d(x,y)= \lambda( (\phi(x) \cup \phi(y) ) \setminus (\phi(x) \cap \phi(y) )$, where $\lambda$ is Lebesgue measure on $\Re$. Furthermore, we have that
\[
\diam_i \{x_i\} = \lambda ( \cup_i \phi(x_i) \setminus \cap_i \phi(x_i) ).
\]
To see that this is true, we consider three cases. We let $x_m$ be the minimum of all $x_i$ and $x_M$ be the maximum.  In case 1, all the $x_i$ are non-negative. Then $\cup_i \phi(x_i) = [0,x_M]$ and $\cap_i \phi(x_i) = [0,x_m]$. This gives the result. In case 2, all the $x_i$ are negative and the result follows similarly. In case 3, some of the $x_i$ are positive and some of the $x_i$ are negative. In this case $\cup_i \phi(x_i) = [x_m,x_M]$ and $\cap_i \phi(x_i)$ is empty.

\begin{prop}
Any $L_1$-embeddable diversity can be embedded in a measure diversity.
\end{prop}
\begin{proof}
Without loss of generality, consider the diversity $(X,\delta_1)$ where $X$ is a subset of $L_1(\Omega,\mathcal{A},\mu)$.
We construct a new measure space  $(X \times \Re,\mathcal{F},\mu \times \lambda)$, i.e. the product measure of $(X,\mathcal{M},\mu)$ with Lebesgue measure on $\Re$.  For $f \in X$ we define $\Phi(f) \subseteq X \times \Re$  by
\[
\Phi(f) = \left\{ (x,y) \in X \times \Re \, | \, y \in \phi(x) \right\}.
\]
We then have that for all finite subsets $\{f_1, \ldots, f_k\}$ of $X$ we have 
\[
\delta\left( \{f_1, \ldots, f_k\} \right) = \delta_{\mu \times \lambda}\left( \{ \Phi(f_1),\ldots,\Phi(f_k) \}\right).
\]
\end{proof}

\subsection{Finite, $L_1$-embeddable diversities}

Further results can be obtained for $L_1$-embeddable diversities $(X,\delta)$ when $X$ is finite, say $|X| = n$. In this case, the study of $L_1$ diversities reduces to the study of non-negative combinations of {\em cut diversities}, also called {\em split diversities}, that are directly analogous to {\em cut metrics}. Given $U \subseteq X$ define the diversity $\delta_U$ by
\[\delta_U(A) = \begin{cases} 1, & \mbox{ if $A \cap U$ and $A \setminus U$ both non-empty}; \\ 0, & \mbox{ otherwise.} \end{cases} \]
In other words, $\delta_U(A) = 1$ when $U$ cuts $A$ into two parts. The set of non-negative combinations of cut diversities for $X$ form a cone which equals the set of $L_1$-embeddable diversities on $X$.

\begin{prop} \label{charact_L1_embed}
Suppose that $|X| = n$ and $(X,\delta)$ is a diversity. The following are equivalent.
\begin{enumerate}
\item[(i)] $(X,\delta)$ is $L_1$-embeddable.
\item[(ii)] $(X,\delta)$ is $\ell_1^m$-embeddable for some $m \leq {n \choose \lfloor n/2 \rfloor}$.
\item[(iii)] $(X,\delta)$ is a {\em split system diversity} (see \cite{Herrmann12}). That is, $\delta$ is a non-negative combination of cut diversities. \end{enumerate}
\end{prop}
 \begin{proof}
(i)$\Rightarrow$(iii)\\ Let $\phi:x \mapsto f_x$ be an embedding from $X$ to $L_1(\Omega,\mathcal{A},\mu)$. For each $U \subseteq X$ and each $\omega \in \Omega$ let
\[\lambda(U,\omega) = \min\{f_u(\omega) - f_v(\omega):u \in U, v \in X \setminus U\}\]
letting $\lambda(U,\omega) = 0$ if this is negative. Define
\[\lambda(U) = \int_\Omega \lambda(U,\omega) \, d\mu(\omega).\]
Then for all $\omega$ and all $A \subseteq X$ we have
\[\diam\{ f_a(\omega):a \in A \} = \sum_U \lambda(U,\omega) \delta_U(A) \]
and so
\[\delta_1(A) = \int_\Omega \diam\{ f_a(\omega):a \in A \} \, d\mu(\omega) = \sum_U \lambda(U) \delta_U(A).\]
(iii) $\Rightarrow$ (ii).\\
Fix $x_0\in X$. We can write $\delta$ as 
\[
\delta(A) = \sum_U \lambda_U \delta_U(A)
\]
for all $A \subseteq X$ where $U$ runs over all subsets of $X$ containing $x_0$.  This collection of subsets of $X$ can be partitioned into $m= {n \choose \lfloor n/2 \rfloor}$ disjoint chains by Dilworth's theorem. Denote these chains by $C_1, \ldots, C_m$ so that
\[
\delta(A) = \sum_{i=1}^m \sum_{U \in C_i} \lambda_U \delta_U (A).
\]
We will show that for every chain $C = C_i$  the diversity 
\[
\delta^C(A) = \sum_{U \in C} \lambda_U \delta_U (A)
\]
is $\mathbb{R}$-embeddable. The result follows. To this end, define $\phi \colon X \rightarrow \mathbb{R}$ by
\[
\phi(x) = \delta^C( \{ x_0, x\}) = \sum_{U \subseteq U_x} \lambda_U 
\]
where $U_x$ is the minimal element of the chain $C$ that contains $x$.
Then 
\begin{eqnarray*}
\delta_1(\phi(A)) & = & \diam \{ \phi(a) \colon a \in A\} \\ 
& = & \diam \{  \delta^C( \{ x_0, a\}) \colon a \in A\} \\
& = & \max_{a \in A} \sum_{U \subseteq U_a} \lambda_U - \min_{a \in A} \sum_{U \subseteq U_a} \lambda_U  \\
& =& \sum_{U \in C} \lambda_U \delta_U(A) \\
&=& \delta^C(A).
\end{eqnarray*}
 
\noindent (ii) $\Rightarrow$ (i).\\
Follows from the fact that $\ell_1^m$ is itself an $L_1$ diversity.
\end{proof}

Diversities formed from combinations of split diversities were studied by \cite{Herrmann12} and in literature on phylogenetic diversities \cite{Moulton07,Spillner08,Minh09}. Proposition~\ref{prop:finiteL1} is a restatement of Theorems 3 and 4 in \cite{BryantKlaere}.

\begin{prop} \label{prop:finiteL1}
Let $(X,\delta)$ be a finite, $L_1$-embeddable diversity, where for all $A \subseteq X$, 
\[\delta(A) = \sum_{U \subseteq X} \lambda_U \delta_U(A),\]
where we assume $\lambda_U = \lambda_{(X \setminus U)}$. For all $A \subseteq X$ we have the identity
\begin{align}
\delta(A) &= \sum_{B \subseteq A} (-1)^{|B|} \delta(B)  \label{eq:FiniteRecurse}
\end{align}
and if $\emptyset \neq A \neq X$ we have
\begin{align}
\lambda_A &= \frac{1}{2} \sum_{B:A\subseteq B}(-1)^{|A|-|B|+1} \delta(B).
\end{align}
\end{prop}

From these we obtain the following characterization of  finite, $L_1$-embeddable metrics.

\begin{prop}
A finite diversity $(X,\delta)$ is $L_1$-embeddable if and only if it satisfies \eqref{eq:FiniteRecurse} and 
\begin{equation}
 \sum_{B:A\subseteq B}(-1)^{|A|-|B|+1} \delta(B) \geq 0 \label{eq:FiniteL1Splits}
\end{equation}
for all $A \subseteq X$, such that $\emptyset \neq A \neq X$.
\end{prop}
\begin{proof}
Necessity follows from Proposition~\ref{prop:finiteL1}. For sufficiency, observe that the map from a weight assignment $\lambda$ to a diversity $\sum_{U \subseteq X} \lambda_U \delta_U$ is linear and, by Proposition~\ref{prop:finiteL1}, invertible for the space of weight functions $\lambda$ satisfying $\lambda_U = \lambda_{X \setminus U}$ for all $U$. The image of this map therefore has dimension $2^{n-1}-1$. From \eqref{eq:FiniteRecurse} we that the diversities $\delta(A)$ for $|A|$ odd can be written in terms of diversities $\delta(A)$ for $|A|$ even. Hence the space of diversities satisfying  \eqref{eq:FiniteRecurse} has dimension $2^{n-1}-1$ and lies in the image of the map. Condition~\ref{eq:FiniteL1Splits} ensures that the diversity is given by a {\em non-negative} combination of cut diversities.
\end{proof}

\section{Minimal-distortion embedding of diversities}

Given two metric spaces $(X_1,d_1)$ and $(X_2,d_2)$ we can ask what is the minimal distortion embedding of $X_1$ into $X_2$, where the minimum is taken over all maps $\phi \colon X_1 \rightarrow X_2$. Naturally, we can ask the same question for diversities. Whereas the question for metric spaces is well-studied (though still containing many interesting open problems) the situation for diversities is almost completely unexplored. We state some preliminary bounds here, most of which leverage on metric results. We begin by proving bounds for several types of diversities defined on $\Re^k$.

\begin{lem} \label{lem:divRk}
Let $\delta^{(1)}_\diam$ and $\delta^{(2)}_\diam$ be the diameter diversities on $\Re^k$, evaluated using $\ell_1$ and $\ell_2$ metrics respectively. Let $\delta_1$ and $\delta_w$ be the $\ell_1$ and mean-width diversities on $\Re^k$. Then for all finite $A \subset \Re^k$
\begin{align}
\delta^{(1)}_\diam(A) \leq \delta_1(A) \leq k \delta^{(1)}_\diam(A) \label{eq:L1bound} \\
\delta^{(2)}_\diam(A) \leq \delta_w(A) \leq \mathcal{O}(\sqrt{k}) \delta^{(2)}_\diam(A) \label{eq:MWbound} 
\end{align}
All bounds are tight. 
\end{lem}

\begin{proof}
The inequalities $\delta^{(1)}_\diam(A) \leq \delta_1(A)$ and $\delta^{(2)}_\diam(A) \leq \delta_w(A)$ are due to 
Theorem~\ref{DivBounds}. 

To prove the $\ell_1$ bounds, note that for each dimension $i$ there are $a^{(i)},b^{(i)} \in A$ which maximize $|a_i - b_i|$. Hence
\begin{align*}
\delta_1(A) & = \sum_{i=1}^k \max\{|a_i - b_i|:a,b \in A\} \\
& = \sum_{i=1}^k |a^{(i)}_i - b^{(i)}_i| \\
& \leq \sum_{i=1}^k d_1(a^{(i)}, b^{(i)}) \\
& \leq k \delta^{(1)}_\diam(A) 
\end{align*}
with equality given by subsets of $\{ \pm e_i:i=1,\ldots,k\}$.

To prove the mean-width bound note that, by Jung's theorem \cite{Danzer63}, a set of points in $\Re^k$ with diameter $d=\delta^{(2)}_\diam(A)$ is contained in some sphere with radius $r$, where 
\[r \leq d \sqrt{\frac{k}{2(k+1)}} \leq \frac{d \sqrt{2}}{2}.\]
Hence $\conv{A}$ is contained in a set with mean width $2r \leq d \sqrt{2}$. From \eqref{eq:widthConstant} we have
\[\delta_w(A) \leq d \sqrt{2} \pi \frac{\mu(S^{n-1})}{\mu(S^n)} =  \frac{\sqrt{2} \pi}{B(k/2,1/2)}  \delta^{(2)}_\diam(A) = \mathcal{O}(\sqrt{k}) \delta^{(2)}_\diam(A),\]
where again $B(\cdot,\cdot)$ denotes the beta function.
The bound holds in the limit for points distributed on the surface of a sphere.
\end{proof}

We now investigate upper bounds for the distortion of diversities into $L_1$ space. To begin, we consider only diversities which are themselves diameter diversities. In many senses, these diversities are similar to metrics, and it is perhaps no surprise that they can embedded with a similar upper bound as their metric counterparts.

\begin{prop}
Let $(X,d)$ be a metric space, $|X| = n$, and let $(X,\dd)$ be the corresponding diameter diversity. 
\begin{enumerate}
\item There is an embedding of $(X,\dd)$ in $\ell_1^k$ with distortion $\mathcal{O}(\log^2 n )$ and  
$k=\mathcal{O} (\log n)$.
\item There is an embedding of $(X,\dd)$ in $(\Re^k,\delta_w)$ with distortion $\mathcal{O}(\log^{3/2} n)$
and $k=\mathcal{O} (\log n )$.
\end{enumerate}
\end{prop}
\begin{proof}
1. Any metric on $n$ points can be embedded into the metric space $\ell_1^k = (\Re^k,d_1)$ with distortion $\mathcal{O}(\log n)$, where $k = \mathcal{O}(\log n)$ \cite{Linial95}.  Let $\phi$ be an embedding for $(X,d)$
with $d(x,y) \leq d_1(\phi(x),\phi(y)) \leq K d_1(x,y)$ for all $x,y \in X$, 
where $K$ is $O(\log n)$. As above, we let $\dd^{(1)}$ denote the diameter diversity for the $\ell_1^k$ metric. 
For all $A \subseteq X$ we have from Lemma~\ref{lem:divRk} that
\[\dd(A) \leq \dd^{(1)}(\phi(A)) \leq \delta_1(\phi(A)) \leq k \dd^{(1)}(\phi(A)) \leq k\cdot K \dd(A).\]
The result now follows since $k$ is $\mathcal{O}(\log n)$ and $K$ is $\mathcal{O}(\log n)$.

2. As shown in \cite{Linial95} (see also \cite{Bourgain85}), there is an embedding $\phi$ of $(X,d)$ into $\ell_2^k$ with
\[d(x,y) \leq d_2(\phi(x),\phi(y)) \leq K d(x,y)\]
for all $x,y \in X$, where $k$ and $K$ are $\mathcal{O} ( \log n )$. 
For all $A \subseteq X$ we have from Lemma~\ref{lem:divRk} that
\[\dd(A) \leq \dd^{(2)}(\phi(A)) \leq \delta_w(\phi(A)) \leq \mathcal{O}(\sqrt{k}) \dd^{(2)}(\phi(A)) \leq \mathcal{O}(\sqrt{k})\cdot K \dd(A).\]
The result follows.
\end{proof}

We now consider the problem of embedding general diversities. The bounds we obtain here can definitely be improved: we do little more than slightly extend the results for diameter diversities. 

\begin{thm} \label{thm:upperb}
Let $(X,\delta)$ be a diversity with $|X|=n$.  
\begin{enumerate}
\item
$(X,\delta)$ can be embedded in $\ell_1^k$ with $k=\mathcal{O} (\log n )$ with distortion $\mathcal{O}(n \log^2 n)$.
\item 
$(X,\delta)$ can be embedded in $\ell_2^k$ with $k=\mathcal{O} (\log n )$ with distortion $\mathcal{O}(n  \log^{3/2} n)$.
\end{enumerate}
\end{thm}
\begin{proof}
Any diversity can be approximated by the diameter diversity of its induced metric with distortion $n$, as shown in Theorem~\ref{DivBounds}. This fact together with the previous theorem gives the required bounds.
\end{proof}

From upper bounds we switch to lower bounds. Any embedding of diversities with distortion $K$ induces an embedding of the underlying metric with distortion at most $K$. Hence we can use the examples from metrics 
\cite{Leighton88} to establish that there are diversities which cannot be embedded in $\ell_1$ with better than an  $\Omega(\log n)$ distortion. 

We have been able to obtain slightly tighter lower bounds for embeddings into $\ell_1^k$ where $k$ is bounded. 

\begin{prop} \label{prop:noBourgain}
Let $(X,\delta)$ be the $n$-point diversity with $\delta(A) = |A|-1$ for all non-empty $A \subseteq X$. Then the minimal distortion embedding of $(X,\delta)$ into $\ell_1^k$ has distortion at least $(n-1)/k$.
\end{prop}
\begin{proof}
For any embedding $\phi$ of $(X,\delta)$, Lemma~\ref{lem:divRk} shows that 
\[
\delta_1(\phi(X)) \leq k \, \diam(\phi(X)) = k\, d(\phi(a) , \phi(b))
\]
for some $a,b \in X$, $a \neq b$.
The distortion of $\phi$ is  equal to 
\[
\max_{A \subseteq X} \frac{\delta(A)}{\delta_1(\phi(A))} \cdot \max_{B \subseteq X} \frac{ \delta_1(\phi(B))}{\delta(B)}.
\]
Taking $A=X$ and $B=\{a,b\}$ shows that the distortion is at least 
 $(n-1)/k$.
\end{proof}

A consequence of Proposition~\ref{prop:noBourgain} is that there will, in general, be no embedding of diversities in $\ell_1$ for which both the distortion and dimension is $\mathcal{O}(\log n)$, or indeed polylog, ruling out a direct translation of the classical embedding results for finite metrics. Even so, we suspect that the upper bounds achieved in Theorem~\ref{thm:upperb} can still be greatly improved.

\section{The geometry of hypergraphs}

Having reviewed diversities, $\ell_1$ diversities, and the diversity embedding problems, we return to their application in combinatorial optimization. We will here establish analogous results to those of  \cite{Linial95} and \cite{Gupta04} for hypergraphs and diversity embeddings into $\ell_1$.  We first state the extensions of maximum multicommodity flows and minimum cuts a little more formally.

Given a hypergraph $H=(V,E)$, non-negative weights $C_e$ for $e\in E$ and  $S \subseteq V$, the goal is find the maximum weighted sum of  minimal connected sub-hypergraphs covering $S$ without exceeding  the capacity of any hyperedge. Let $\mathcal{\mathcal{T}}_S$ be set of all minimal connected sub-hypergraphs of $H$ that include $S$.  For each sub-hypergraph $t\in \mathcal{T}_S$ assign weight  $z_t$. We consider the following generalization of {\em fractional Steiner tree packing} \cite{Jain03} which we call \emph{maximum hypergraph Steiner packing}: Identify $z_t$ satisfying the LP:
\begin{equation*}
\begin{array}{rrcll}
\mbox{maximize} & \sum_{t \in  \mathcal{T}_{S} } z_t & &  \\
\mbox{subject to} &\sum_{t \in \mathcal{T}_{S} : e \in t} z_t & \leq & C_e &  \mbox{for all $e\in E$},  \\
& z_t & \geq & 0, &  \mbox{for all $t \in \mathcal{T}_{S}$}.
\end{array}
\end{equation*} 
As before, if we define $C_e$ for all subsets $e$ of $V$, and let it be zero for $e \not\in E$, we can  drop the dependence of the problem on $E$. The reference \cite{Kiraly08} studies an oriented version of this problem.

As with flows, maximum hypergraph Steiner packing has a multicommodity version. For each subset $S$ of $V$ suppose we have  non-negative demand $D_S$.
We view $D$ and $C$ as non-negative vectors indexed by all subsets of $V$.
Suppose we want to simultaneously connect up  all $S \subseteq V$ with minimal connected sub-hypergraphs carrying flow $f D_S$ for all $S \subseteq V$ and we want to maximize $f$. The corresponding optimization problem is:
\begin{equation} \label{eqn:maxflowLP}
\begin{array}{rrcll}
 \mbox{maximize} &  f & & \\
\mbox{subject to}  &  \sum_S \sum_{t \in \mathcal{T}_{S} \colon R \in t} z_{t,S} & \leq & C_R, & \mbox{for all }R \subseteq V, \\
& \sum_{t \in \mathcal{T}_{S}} z_{t,S} & = & f \cdot D_S, & \mbox{for all } S\subseteq V,  \\
&  z_{t,S} & \geq & 0, & \mbox{for all } S\subseteq V,  t \in \mathcal{T}_{S}.
\end{array}
\end{equation} 
Note that we use $z_{t,S}$ rather than just $z_t$ because the same connected sub-hypergraph  $t$ might cover more than one set $S$ in the hypergraph. We call the optimal value of $f$ for this problem $\mbox{MaxHSP}(V,C,D)$, for \emph{maximum multicommodity hypergraph Steiner packing}.

Next we define the appropriate analogues of the min-cut problem, which we call \emph{minimum hypergraph  cut}. As before, we let  $\partial U$  be the set of hyperedges which have endpoints in both $U$ and $V \setminus U$, and we make the simplifying assumption that every subset is a hyperedge, including any missing hyperedges with capacity zero. We define
\[
\mbox{MinHypCut}(V,C,D) = \min_{U \subseteq V} \frac{ \sum_{A \in \partial U }  C_A }
{\sum_{S \in \partial U} D_S}.\]

Below we will show that 
\(
\mathrm{MaxHSP}(V,C,D) \leq \mathrm{MinHypCut}(V,C,D).
\)
We define
\[
\gamma(V,C,D) = \frac{\mathrm{MinHypCut}(V,C,D)}{\mathrm{MaxHSP}(V,C,D)}.
\]
We say that a non-negative vector $C$  is \emph{supported on the hypergraph} $H=(V,E)$ if  $C_e = 0$ for $e \not\in E$.
Then for any hypergraph $H$  we define $\gamma(H)$ to be the greatest value of $\gamma(V,C,D)$ over all nonnegative $C$ and $D$ such that $C$ is supported on $H$.
 
  We say that a diversity $\delta$ on $V$ is \emph{supported on} $H=(V,E)$ if it is the hypergraph Steiner diversity of $H$ for some set of non-negative weights $C_e$ for $e \in E$. 
For any diversity $\delta$ on $V$ we define $k_1(\delta)$ to be the minimal distortion between $\delta$ and an $\ell_1$-embeddable diversity on $V$. For any hypergraph $H$ we define $k_1(H)$ to be the maximum of $k_1(\delta)$ over all diversities $\delta$ supported on $H$.  The major result for this section is that  for all hypergraphs $H$
 \[
 k_1(H) = \gamma(H).
 \]
 The fact that $\gamma(H) \leq k_1(H)$  (our Theorem 1)  is the analogue of  results in Section 4 of  \cite{Linial95} and the fact that equality holds  (our Theorem 2) is the analogue  of  Theorem 3.2 in \cite{Gupta04}.

 \begin{prop} \label{prop:characterizeMaxHSP}
 For all $V,C,D$, 
 \[
{\mathrm{MaxHSP}}(V,C,D)=  \min_{\delta \in \Delta(V)} \frac{ C \cdot \delta}{ D \cdot \delta}
\]
where $\Delta(V)$ is the set of all diversities on $V$. In particular, the optimal $\delta \in \Delta(V)$ is supported on the hypergraph $H=(V,E)$ where $E$ is  the set of all  $e$ such that $C_e>0$.
\end{prop}

\begin{proof}

We rewrite the linear program \eqref{eqn:maxflowLP} in standard form. We break the equality constraint into $\leq$ and $\geq$ and note that we can omit the $\geq$ constraint, because it will never be active. Then we get
\begin{equation} \label{eqn:maxflowStandard}
\begin{array}{rrcll}
 \mbox{maximize }&  f &  & & \\
\mbox{subject to} & \sum_S \sum_{t \in \mathcal{T}_{S} : R\in t} z_{t,S} & \leq & C_R, & \mbox{for all }R \subseteq V, \\
& \left( \sum_{t \in \mathcal{T}_{S}} -z_{t,S} \right) + f \cdot D_S & \leq  & 0, & \mbox{for all } S \subseteq V, \\
& z_{t,S} & \geq & 0, & \mbox{for all } S \subseteq V,  t \in \mathcal{T}_{S}.
\end{array}
\end{equation} 
Let $d_R$ be the dual variables corresponding to the first set of inequality constraints, and let $y_S$ be the inequalities corresponding to the second set of inequality constraints. Then  the dual problem is
\begin{equation} \label{eqn:maxflowDual}
\begin{array}{rrcll}
 \mbox{minimize} & \sum_{R \subseteq V} C_R d_R & & \\
\mbox{subject to} & \sum_{S} y_S D_S  & \geq  & 1 \\
& \sum_{R \in t} d_R & \geq & y_S, & \mbox{for all }S \subseteq V, t \in \mathcal{T}_{S}, \\
 & d_R & \geq & 0, & \mbox{for all }R \subseteq V, \\
 & y_S & \geq & 0, & \mbox{for all }S \subseteq V.
\end{array}
\end{equation}
By strong duality, \eqref{eqn:maxflowStandard} and \eqref{eqn:maxflowDual} have the same optimal values.
Next we show that  \eqref{eqn:maxflowDual}  is equivalent to 
\begin{equation} \label{eqn:maxflowDiversity}
\begin{array}{rrcll}
\mbox{minimize} &  \sum_{R \subseteq V} C_R \delta(R), & & \\
\mbox{subject to} & \sum_{S \subseteq V} D_S  \delta(S) & \geq & 1, \\
& \delta \mbox{ is a diversity.} & & 
\end{array}
\end{equation}
where the minimum is taken over all diversities. 

To see the equivalence of \eqref{eqn:maxflowDual} and   \eqref{eqn:maxflowDiversity}, suppose that $\delta$ is a diversity solving \eqref{eqn:maxflowDiversity}. Let $d_R = \delta(R)$ for all $R \subseteq V$ and $y_S = \delta(S)$ for all $S \subseteq V$.
Then the objective function of \eqref{eqn:maxflowDual} is the same, the second line of \eqref{eqn:maxflowDual} still holds,
the third line holds by the triangle inequality for diversities, and the fourth and fifth line hold by the non-negativity of diversities.

To see the other direction, suppose $d_R$ and $y_S$ solve \eqref{eqn:maxflowDual}. Let $\delta$ be the Steiner diversity on $V$  generated edge weights $d_R$, $R \subseteq V$. Since $\delta(R) \leq d_R$ for all $R$, this can only decrease the objective function. 
Also, by the definition of $\delta$, $\delta(S) \geq y_S$ for all $S$, so the inequality of \eqref{eqn:maxflowDiversity} is satisfied too. Thus the two LPs have the same minima.

Note we can assume that $\delta$ is the Steiner diversity for a weighted hypergraph with hyperedges $\{ R \colon C_R >0\}$. 
If not, we can replace $\delta$ with the Steiner diversity on the hypergraph whose hyperedges are the set $\{ R \colon C_R >0\}$ and whose weights are the $C_R$. This Steiner diversity will have the same value on the hyperedges as $\delta$, so the objective function will not change, but the value can only increase on other subsets of $V$, and so the constraint is still satisfied.

Finally, \eqref{eqn:maxflowDiversity} is equivalent to 
\begin{equation} \label{eqn:maxflowFinal}
\begin{array}{rrcll}
\mbox{minimize} &  \frac{ \sum_{R \subseteq V} C_R \delta(R)  }{ \sum_{S \subseteq V} D_S \delta(S)}, && \\
\mbox{subject to} & \delta \mbox{ is a diversity.} & &
\end{array}
\end{equation}
This is because, any solution of \eqref{eqn:maxflowDiversity} will only have a smaller or equal value for the objective function of \eqref{eqn:maxflowFinal}. And any solution of \eqref{eqn:maxflowFinal} can be rescaled without changing the objective function so that 
$\sum_{S \subseteq V} D_S  \delta(S) = 1$, giving a feasible solution to \eqref{eqn:maxflowDiversity} with the same objective function.
This rescaling will not change the hypergraph that $\delta$ is supported on.
\end{proof}

\begin{prop} For all $V, C, D$,
\[
\mathrm{MinHypCut}(V,C,D)=  \min_{\delta \in \Delta_1(V)}  \frac{ C \cdot \delta}{ D \cdot \delta}
\] 
where $\Delta_1(V)$ is the set of all $\ell_1$-embeddable diversities on $V$.
 \end{prop}
 
 \begin{proof}
 For any cut $(U, V \setminus U)$ of $V$,  let $\delta_U$ be the corresponding cut diversity. Then by definition we have that
 \[
 \mathrm{MinHypCut}(V,C,D)=\min_{U \subseteq V}  \frac{ C \cdot \delta_U}{ D \cdot \delta_U},
 \]
 where we restrict $U$ to values where the denominator is non-zero.
 We  need to show that this value is not decreased by taking the minimum over all $\ell_1$-embeddable diversities instead.
 
 Let $\delta$ be $\ell_1$-embeddable diversity that minimizes the ratio. By Proposition \ref{charact_L1_embed}, $\delta$ can be expressed as a finite linear combination of cut-diversities: $\delta = \sum_i a_i \delta_{U_i}$ for some non-negative $a_i$ and some subsets  $U_i$ of $V$. Let $I$ be the index $i$ that minimizes $C \cdot \delta_{U_i} / D \cdot \delta_{U_i}$. Then we claim that
 \[
 \frac{ C \cdot \delta}{ D \cdot \delta} \geq \frac{ C \cdot \delta_{U_I}}{ D \cdot \delta_{U_I}}
 \]
 To see this, observe that
 \[
 \frac{ C \cdot \delta}{ D \cdot \delta} = \frac{\sum_i  a_i  (C \cdot \delta_{U_i})}{\sum_i a_i (D \cdot \delta_{U_i})}
 \]
 
 Let
 \[
 G(a) = \frac{a \cdot x}{a \cdot y}
 \]
 for vectors $a$ with $a_i \geq 0$ for all $i$,
 where $x$ and $y$ are non-negative vectors of the same size. We claim that $G$ attains its minimum on this domain at a value of $a$ consisting of a vector with a single non-zero entry. To show this, we compute the gradient of $G$
 \[
 \nabla G(a) = \frac{1}{(a \cdot y)^2} [ x (a \cdot y) - y (a \cdot x) ].
 \]
 If $x$ and $y$ are parallel then the result immediately follows so assume that they are not. Then $\nabla G$ is not zero anywhere in the domain, and so the maximum of $G$ must be taken on boundary of the domain. So at least one $a_i$ must be zero. Discard this term from the numerator and the denominator of $G$. Then repeat the argument for $G$ as a function of a vector of one fewer entries. Repeating gives a single non-zero value, which may be set to 1.
 \end{proof}
 
 The following theorem implies Theorem~\ref{major1}.
 
 \begin{thm} \label{thm:kiwi} For all hypergraphs $H=(V,E)$, non-negative hyperedge capacities $C$ supported on $H$, and non-negative $D$
 \[
 \mathrm{MaxHSP}(V,C,D)  \leq \mathrm{MinHypCut}(V,C,D) \leq k_1(H) \mathrm{MaxHSP}(V,C,D).
 \]
 \end{thm}
 \begin{proof}
 Since $\Delta_1(V) \subseteq \Delta(V)$, the first inequality follows from the previous two results.

For the second inequality, given $V,C,D$ and hypergraph $H=(V,E)$ supporting $C$, let $\delta$ solve the MaxHSP linear program \eqref{eqn:maxflowDiversity}. By Proposition~\ref{prop:characterizeMaxHSP}  we know that  $\delta$ is supported on $H$.  Let $\hat{\delta}$ be the minimal-distortion $\ell_1$ embeddable 
diversity of $\delta$. We may assume that $\delta \leq \hat{\delta} \leq k_1(H) \delta$. Then
\[
\mathrm{MinHypCut}(V,C,D) \leq 
\frac{C \cdot \hat{\delta}}{D \cdot \hat{\delta}} \leq k_1(H) \frac{C \cdot \delta}{D \cdot \delta} = k_1(H) \mathrm{MaxHSP}(V,C,D)
\]
as required.
\end{proof}

The following theorem implies Theorem~\ref{major2}.
	
\begin{thm} \label{thm:hyrax}
For all hypergraphs $H$ there exist weights $C, D$ with $C$ supported on $H$ such that 
\[
k_1(H) = \frac{\mathrm{MinHypCut}(V,C,D)}{\mathrm{MaxHSP}(V,C,D) }.
\]
Hence the upper bound in Theorem~\ref{thm:kiwi} is tight.
\end{thm}
	
To prove this result, we will need a lemma from \cite{Gupta04} which we reproduce here.  
\begin{lem}[Claim A.2 of \cite{Gupta04}]
Let $v,u \in \Re^k$ be positive vectors. Define 
\[
H(v,u) = \max_i \frac{u_i}{v_i} \cdot \max_j \frac{v_j}{u_j}.
\]
If $S \subseteq \Re^k$ is a closed set of positive vectors, define $H(v,S)$ as $\min_{u\in S} H(v,u)$.
If $K \subset \Re^k$ is a closed convex cone, then 
\[
H(v,K) = \max_{C,D} \frac{D \cdot v}{C \cdot v},
\]
where the maximum is taken over all non-negative vectors $D,C \in \Re^k$ for which $\frac{D \cdot u}{C \cdot u} \leq 1$ for any $u \in K$.
\end{lem}

\begin{proof} (of  Theorem \ref{thm:hyrax})
Let $\delta$ be a diversity supported by the hypergraph $H$ that maximizes $k_1(\delta)$, and define $\lambda = k_1(\delta) = k_1(H)$.
We need to show that 
\[
\lambda \leq  \max_{C,D} \frac{\mathrm{MinHypCut}(V,C,D)}{\mbox{MaxHSP}(V,C,D)},
\]
where the maximum is taken over all $C,D$ where $C$ is supported on $H$.
 
Let $v$ be given by $\delta$, and let $K$ be the cone of all $\ell_1$-embeddable diversities on $V$. Then $\lambda = H(v,K)$.  We  apply the lemma to show that 
 \[
\lambda= \max_{C,D} \frac{D \cdot \delta}{C \cdot \delta},
\]
where the maximum is taken over all non-negative vectors $C,D$  which satisfy the restriction $\frac{D \cdot \mu}{C \cdot \mu} \leq 1$ for any $\ell_1$-embeddable diversity $\mu$.  This tells us that there exists $C,D$ such that $\lambda = \frac{D \cdot \delta}{C \cdot \delta}$ and $\frac{D \cdot \mu}{C \cdot \mu} \leq 1$ for any $\ell_1$-embeddable diversity $\mu$.

  First we show that we may assume that $C$ is supported on $H$. Suppose that for some set $R \subseteq V$, 
  $R \not \in E$ we have $C_R>0$. Since $\delta$ is supported on $H$ there are hyperedges $h_1, \ldots, h_k$ that form a connected set covering $R$ with $\delta(R) = \sum_{i=1,\ldots, k} \delta(h_k)$. Define a new vector $C'$ by 
\begin{eqnarray*}
C'_R & = & 0, \\
C'_{h_i} & = & C_{h_i} + C_R, \mbox{ for } i = 1, \ldots, k, \mbox{and} \\
C'_S & = & C_S, \mbox{ otherwise.}
\end{eqnarray*}
Even with this new $C',D$ we still have
$\lambda=  \frac{D \cdot \delta}{C' \cdot \delta}$ and  $\frac{D \cdot \mu}{C' \cdot \mu} \leq 1$ for any $\ell_1$-embeddable diversity $\mu$. To see this,  first note that $C' \cdot \mu \geq C \cdot \mu$ so 
\[
\frac{D \cdot \mu}{C' \cdot \mu}  \leq \frac{D \cdot \mu}{C \cdot \mu} \leq 1.
\]
Secondly, since $\delta$ satisfies $\delta(R) = \sum_{i=1,\ldots, k} \delta(h_k)$ and these are the only sets on which $C$ is changed, it follows that $C' \cdot \delta = C \cdot \delta$.  We repeat this procedure until we have $C_R <0$ only if $R \in E$.  

Using this $C$ and $D$ gives us
\[
\lambda = \frac{D \cdot \delta}{C \cdot \delta} \leq \frac{ \min_{\mu \in \Delta_1(V)} (C \cdot \mu)/(D \cdot \mu)   }{ ( C \cdot \delta)/(D \cdot \delta)  } \leq \frac{\mathrm{MinHypCut}(V,C,D)}{\mathrm{MaxHSP}(V,C,D)}.
\]
\end{proof}

\bibliographystyle{abbrvnat}

\begin{thebibliography}{23}
\providecommand{\natexlab}[1]{#1}
\providecommand{\url}[1]{\texttt{#1}}
\expandafter\ifx\csname urlstyle\endcsname\relax
  \providecommand{\doi}[1]{doi: #1}\else
  \providecommand{\doi}{doi: \begingroup \urlstyle{rm}\Url}\fi

\bibitem[Abramowitz and Stegun(1964)]{Abramowitz64}
M.~Abramowitz and I.~A. Stegun.
\newblock \emph{Handbook of mathematical functions with formulas, graphs, and
  mathematical tables}, volume~55 of \emph{National Bureau of Standards Applied
  Mathematics Series}.
\newblock For sale by the Superintendent of Documents, U.S. Government Printing
  Office, Washington, D.C., 1964.

\bibitem[Bourgain(1985)]{Bourgain85}
J.~Bourgain.
\newblock On {L}ipschitz embedding of finite metric spaces in {H}ilbert space.
\newblock \emph{Israel J. Math.}, 52\penalty0 (1-2):\penalty0 46--52, 1985.
\newblock ISSN 0021-2172.
\newblock \doi{10.1007/BF02776078}.

\bibitem[Bryant and Klaere(2012)]{BryantKlaere}
D.~Bryant and S.~Klaere.
\newblock The link between segregation and phylogenetic diversity.
\newblock \emph{J. Math. Biol.}, 64\penalty0 (1-2):\penalty0 149--162, 2012.
\newblock ISSN 0303-6812.
\newblock \doi{10.1007/s00285-011-0409-5}.

\bibitem[Bryant and Tupper(2012)]{Bryant12}
D.~Bryant and P.~F. Tupper.
\newblock Hyperconvexity and tight-span theory for diversities.
\newblock \emph{Advances in Mathematics}, 231\penalty0 (6):\penalty0 3172 --
  3198, 2012.
\newblock ISSN 0001-8708.
\newblock \doi{http://dx.doi.org/10.1016/j.aim.2012.08.008}.

\bibitem[Danzer et~al.(1963)Danzer, Gr{\"u}nbaum, and Klee]{Danzer63}
L.~Danzer, B.~Gr{\"u}nbaum, and V.~Klee.
\newblock Helly's theorem and its relatives.
\newblock In \emph{Proc. {S}ympos. {P}ure {M}ath., {V}ol. {VII}}, pages
  101--180. Amer. Math. Soc., Providence, R.I., 1963.

\bibitem[Deza and Laurent(1997)]{Deza97}
M.~M. Deza and M.~Laurent.
\newblock \emph{Geometry of cuts and metrics}, volume~15 of \emph{Algorithms
  and Combinatorics}.
\newblock Springer-Verlag, Berlin, 1997.
\newblock ISBN 3-540-61611-X.

\bibitem[Esp{\'{\i}}nola and Pi{\c{a}}tek(2014)]{Espinola14}
R.~Esp{\'{\i}}nola and B.~Pi{\c{a}}tek.
\newblock Diversities, hyperconvexity and fixed points.
\newblock \emph{Nonlinear Anal.}, 95:\penalty0 229--245, 2014.
\newblock ISSN 0362-546X.
\newblock \doi{10.1016/j.na.2013.09.005}.

\bibitem[Faith(1992)]{Faith92}
D.~Faith.
\newblock Conservation evaluation and phylogenetic diversity.
\newblock \emph{Biological Conservation}, 61:\penalty0 1--10, 1992.

\bibitem[Gupta et~al.(2004)Gupta, Newman, Rabinovich, and Sinclair]{Gupta04}
A.~Gupta, I.~Newman, Y.~Rabinovich, and A.~Sinclair.
\newblock Cuts, trees and {$l_1$}-embeddings of graphs.
\newblock \emph{Combinatorica}, 24\penalty0 (2):\penalty0 233--269, 2004.
\newblock ISSN 0209-9683.
\newblock \doi{10.1007/s00493-004-0015-x}.

\bibitem[Herrmann and Moulton(2012)]{Herrmann12}
S.~Herrmann and V.~Moulton.
\newblock Trees, tight-spans and point configurations.
\newblock \emph{Discrete Math.}, 312\penalty0 (16):\penalty0 2506--2521, 2012.
\newblock ISSN 0012-365X.
\newblock \doi{10.1016/j.disc.2012.05.003}.

\bibitem[Jain et~al.(2003)Jain, Mahdian, and Salavatipour]{Jain03}
K.~Jain, M.~Mahdian, and M.~R. Salavatipour.
\newblock Packing {S}teiner trees.
\newblock In \emph{Proceedings of the {F}ourteenth {A}nnual {ACM}-{SIAM}
  {S}ymposium on {D}iscrete {A}lgorithms ({B}altimore, {MD}, 2003)}, pages
  266--274, New York, 2003. ACM.

\bibitem[Kir\'aly and Lau(2008)]{Kiraly08}
T.~Kir\'aly and L.~C. Lau.
\newblock Approximate min-max theorems for {S}teiner rooted-orientations of
  graphs and hypergraphs.
\newblock \emph{Journal of Combinatorial Theory, Series B}, 98\penalty0
  (6):\penalty0 1233 -- 1252, 2008.
\newblock ISSN 0095-8956.
\newblock \doi{http://dx.doi.org/10.1016/j.jctb.2008.01.006}.

\bibitem[Lau(2007)]{Lau07}
L.~C. Lau.
\newblock An approximate max-{S}teiner-tree-packing min-{S}teiner-cut theorem.
\newblock \emph{Combinatorica}, 27\penalty0 (1):\penalty0 71--90, 2007.
\newblock ISSN 0209-9683.
\newblock \doi{10.1007/s00493-007-0044-3}.

\bibitem[Leighton and Rao(1988)]{Leighton88}
T.~Leighton and S.~Rao.
\newblock An approximate max-flow min-cut theorem for uniform multicommodity
  flow problems with applications to approximation algorithms.
\newblock In \emph{Foundations of Computer Science, 1988., 29th Annual
  Symposium on}, pages 422--431, 1988.
\newblock \doi{10.1109/SFCS.1988.21958}.

\bibitem[Linial et~al.(1995)Linial, London, and Rabinovich]{Linial95}
N.~Linial, E.~London, and Y.~Rabinovich.
\newblock The geometry of graphs and some of its algorithmic applications.
\newblock \emph{Combinatorica}, 15\penalty0 (2):\penalty0 215--245, 1995.
\newblock ISSN 0209-9683.
\newblock \doi{10.1007/BF01200757}.

\bibitem[Minh et~al.(2009)Minh, Klaere, and von Haeseler]{Minh09}
B.~Minh, S.~Klaere, and A.~von Haeseler.
\newblock Taxon selection under split diversity.
\newblock \emph{Syst. Biol.}, 58\penalty0 (6):\penalty0 586--594, 2009.

\bibitem[Moulton et~al.(2007)Moulton, Semple, and Steel]{Moulton07}
V.~Moulton, C.~Semple, and M.~Steel.
\newblock Optimizing phylogenetic diversity under constraints.
\newblock \emph{Journal of Theoretical Biology}, 246\penalty0 (1):\penalty0 186
  -- 194, 2007.
\newblock ISSN 0022-5193.
\newblock \doi{http://dx.doi.org/10.1016/j.jtbi.2006.12.021}.

\bibitem[Poelstra(2013)]{Poelstra13}
A.~Poelstra.
\newblock On the topological and uniform structure of diversities.
\newblock \emph{Journal of Function Spaces and Applications}, 2013:\penalty0 9
  pages, 2013.

\bibitem[Saad et~al.(2008)Saad, Terlaky, Vannelli, and Zhang]{Saad08}
M.~Saad, T.~Terlaky, A.~Vannelli, and H.~Zhang.
\newblock Packing trees in communication networks.
\newblock \emph{Journal of combinatorial optimization}, 16\penalty0
  (4):\penalty0 402--423, 2008.

\bibitem[Shephard(1968)]{Shephard68}
G.~C. Shephard.
\newblock The mean width of a convex polytope.
\newblock \emph{J. London Math. Soc.}, 43:\penalty0 207--209, 1968.
\newblock ISSN 0024-6107.

\bibitem[Spillner et~al.(2008)Spillner, Nguyen, and Moulton]{Spillner08}
A.~Spillner, B.~T. Nguyen, and V.~Moulton.
\newblock Computing phylogenetic diversity for split systems.
\newblock \emph{Computational Biology and Bioinformatics, IEEE/ACM Transactions
  on}, 5\penalty0 (2):\penalty0 235--244, 2008.

\bibitem[Steel(2005)]{Steel05}
M.~A. Steel.
\newblock Phylogenetic diversity and the greedy algorithm.
\newblock \emph{Syst. Biol.}, 54\penalty0 (4):\penalty0 527--529, 2005.

\bibitem[Taylor(2006)]{Taylor06}
M.~E. Taylor.
\newblock \emph{Measure theory and integration}, volume~76 of \emph{Graduate
  Studies in Mathematics}.
\newblock American Mathematical Society, Providence, RI, 2006.
\newblock ISBN 978-0-8218-4180-8; 0-8218-4180-7.

\end{thebibliography}

\end{document}